\documentclass[makeidx]{amsart}
\usepackage[active]{srcltx}
\usepackage{amsmath}
\usepackage{color}
\usepackage{amssymb}
\usepackage{epsfig}
\usepackage{hyperref}
\newtheorem{Proposition}{Proposition}
  \newtheorem{Remark}{Remark}
  
  \newtheorem{Lemma}[Proposition]{Lemma}
  \newtheorem{Theo}[Proposition]{Theorem}

 \newtheorem{Definition}[Proposition]{Definition}
 \newtheorem{Note}[Remark]{Note}

\def\blackslug{\hbox{\hskip 1pt \vrule width 4pt height 8pt depth 1.5pt
\hskip 1pt}}
\def\qed{\quad\blackslug\lower 8.5pt\null\par}

\def\Im{\mathrm{Im}}

 \makeindex
\title[Lifetime of shape oscillations]{The lifetime of shape oscillations of a bubble in an unbounded, inviscid and compressible fluid with surface tension}
\author{O. Costin, S. Tanveer,  and M.I. Weinstein}
\address{Department of Mathematics, The Ohio State University, 231 W 18th Ave, Columbus, OH 43210}
\address{Department of Applied Physics and Applied Mathematics, Columbia University, New York, NY 10027}
\begin{document}
\maketitle 

\begin{abstract} General perturbations of a spherical gas bubble in a compressible and inviscid fluid with surface tension were proved in \cite{SW:11}, in the linearized approximation, to decay exponentially, $\sim e^{-\Gamma t},\ \Gamma>0$, as time advances. Formal asymptotic and numerical evidence led to the conjecture that $\Gamma \approx  
\frac{A}{\epsilon}\ \frac{We}{ \epsilon^{2}}\ \exp\left(-B \frac{We}{\epsilon^2}\right)$,
 where $0<\epsilon\ll1$ is the Mach number and $A$ and $B$ are positive constants.
 In this paper, we prove this conjecture and calculate $A$ and $B$ to leading order in $\epsilon$.
 \end{abstract}
 
 \section{Introduction and outline}
The detailed non-spherical deformations of gas bubbles in a liquid is a problem of great physical interest in fundamental and applied physics; see, for example, \cite{Brenner:95,Leighton:04,SW:11} and references cited therein. We 
consider the dynamics of a gas bubble in a compressible, inviscid and irrotational fluid with surface tension.  This physical system has an equilibrium state, consisting of: a spherically symmetric
gas bubble at constant pressure inside the bubble, and a fluid at constant (lower) pressure 
 and vanishing velocity, outside the bubble. 

 In \cite{SW:11} the linearized dynamics about such equilibria was studied and it was shown that general finite energy perturbations of the spherical equilibrium damp out as time advances. In particular, the $L^\infty$ and local energy norms of the perturbation tend to zero exponentially, $\sim e^{-\Gamma t}$, as $t\to\infty$.  Formal asymptotic and numerical evidence led to the conjecture that 
 \begin{equation}\label{formal-asymp}
 \Gamma(\epsilon) \approx  \frac{A}{\epsilon}\ \frac{We}{ \epsilon^{2}}\ \exp\left(-B \frac{We}{\epsilon^2}\right)\ .
 \end{equation}
Here, $0<\epsilon\ll1$ denotes the Mach number, a dimensionless ratio of speeds, and $We$ denotes the Weber number, a dimensionless measure of surface tension. $A$ and $B$ are positive constants determined in \cite{SW:11} by computer simulation.  %
 \footnote{ The Mach number, $\epsilon$, is the ratio of the bubble-wall radial velocity to the sounds speed in the fluid, exterior to the bubble.  We shall later set $We$ equal to one and consider the asymptotics for small $\epsilon$.
}\medskip
 
 The asymptotics \eqref{formal-asymp} are in marked contrast to the decay rate of perturbations which are spherically symmetric:
 \begin{equation}\label{Gamma-radial}
 \Gamma_{\rm radial}(\epsilon) \ =\ \mathcal{O}(\epsilon)
 \end{equation}
 Thus, non-spherical deformations excite {\it shape modes} which lose their energy to the fluid and radiate sound waves very slowly.
 
{\it  Our  goal in this article is to present a proof of the above conjecture and a calculation of  $A$ and $B$ to leading order in $\epsilon$.}
 \bigskip
 
 A  systematic and detailed discussion of the physical problem, the full nonlinear compressible equations and the appropriate linearization  is presented in \cite{SW:11}. 
The dynamics of the linearized velocity potential, $\Psi(x,t),\ x\in\mathbb{R}^3$, and the bubble shape perturbation, $\beta(\Omega,t),\ \Omega\in S^2$ is governed by the wave-system:
\begin{subequations}\label{linearized-eqns}
\begin{align}
\epsilon^2\partial_t^2\Psi\ -\ \Delta\Psi\ &=\ 0,\ \ &r=|x|>1\label{wave}\\
\partial_r\Psi\ &=\ \partial_t\beta,\ \ &r=1\label{kinematic}\\
\partial_t\Psi\ &=\ 3\gamma\left(\frac{1}{2}+\frac{2}{We}\right)\ \left\langle\beta,Y_0^0\right\rangle\ -\ 
 \frac{1}{We}\left(2+\Delta_S\right)\beta,\ \ &r=1\label{LY}\\
 \left\langle\beta,Y_1^m\right\rangle\ & =\ 0\ .\ \ &|m|\le1 \label{CofM}
 \end{align}
 \end{subequations}
 \footnote{ For simplicity, we have set $Ca$, the cavitation number, appear in \cite{SW:11} equal to one.}  
 Equation \eqref{wave} follows from linearization of the  continuity and Euler momentum equations, exterior to the equilibrium bubble. Equation \eqref{kinematic} is the linearization of the kinematic boundary condition. Equation \eqref{LY} is the linearization of the Laplace-Young boundary condition, stating that the pressure jump across the gas - fluid interface is proportional to the mean curvature. Finally, \eqref{CofM} is the linearization of the statement   that the origin of coordinates (the center of the bubble) is in a frame of reference moving with the bubble center of mass. The equilibrium bubble has been normalized to have unit radius. Here, $Y_l^m=Y_l^m(\Omega),\ \Omega\in S^2$ are spherical harmonics, which satisfy $-\Delta_SY_l^m=l(l+1)Y_l^m,\ l\ge0,\ |m|\le l$,\ 
  where $\langle Y_l^m,Y_{l'}^{m'}\rangle_{L^2(S^2)}=\delta_{ll'}\delta_{mm'}$. In particular, $Y_0^0(\Omega)\equiv (4\pi)^{-1}$.
 \medskip
 
 As explained in \cite{SW:11} and now outlined,  the $L^\infty$ and local-energy decay, for $\epsilon$ small,  is controlled by a  non-selfadjoint eigenvalue problem, which arises by seeking time-harmonic solutions of \eqref{wave}-\eqref{CofM}:
 \[ \Psi(\Omega,t)\ =\ e^{- i \lambda t}\Psi_\lambda(r,\Omega),\ \ \beta(\Omega,t)\ =\ e^{-i \lambda t}\beta_\lambda(\Omega)\ , \]
  which satisfy an outgoing radiation condition as $r\to\infty$: \medskip
  
 \noindent  {\bf The Scattering Resonance Spectral Problem} 
 
 \begin{subequations}\label{eq:srp}
\begin{align}
\left(\ \Delta +\left(\epsilon \lambda\right)^2\ \right) \Psi_{\lambda} &= 0, && r>1\label{helmholtz}\\
\partial_{r}\Psi_{\lambda} &= -i\lambda \beta_{\lambda}, && r=1\label{lin-kin}\\
-i\lambda \Psi_{\lambda} &= 3\gamma\left( \frac{1}{2} + \frac{2}{We} \right) \ 
\langle \beta_{\lambda},Y_0^0 \rangle Y_0^0 -
 \frac{1}{We}\ (2+\Delta_{S})\beta_{\lambda}\ , && r=1 \label{lin-LY}\\
\Psi_{\lambda} &\phantom{=} \text{outgoing} && r\to\infty. \label{sommerfeld}
\end{align}
\end{subequations}
If $\lambda$ is such that \eqref{eq:srp} has a non-trivial solution,
then we call $\lambda$ a (deformation) scattering resonance energy, and $(\Psi_\lambda,\beta_\lambda)$ a corresponding scattering resonance mode. Due to the radiation condition \eqref{sommerfeld}, the eigenvalue problem \eqref{eq:srp} is non-selfadjoint. The eigenvalues of \eqref{eq:srp} all lie in the open lower half complex plane and for each $\epsilon>0$ are uniformly bounded away from the real axis. In particular, for any $\epsilon>0$ there is a scattering resonance eigenvalue $\lambda_\star(\epsilon)$ with $\Im\ \lambda_\star(\epsilon)<0$ and such that \underline{any} scattering resonance eigenvalue of \eqref{eq:srp}, $\lambda$, satisfies
\[ \Im\ \lambda\ \le\ \Im\ \lambda_\star(\epsilon)<0\ ;\]
see Theorem 3.2 of \cite{SW:11}. 
\medskip

The following theorem, a consequence of 
Theorem 5.1 in \cite{SW:11},  relates $\lambda_\star(\epsilon)$ to the rate of decay of perturbations about the spherical bubble equilibrium:
 
 \begin{Theo} \label{thm:linsol}
There exists $\epsilon_0>0$, such that for all $0<\epsilon\le\epsilon_0$ the following holds. 
Consider the initial boundary value problem  \eqref{linearized-eqns}. Assume the initial conditions:
\begin{itemize}
\item[(a)] $\beta(t,\Omega)\ =\ \sum_{l\ge0}\sum_{|m|\le l} \beta_l^m(0)\ Y_l^m(\Omega)$, where
\[
\| \beta(t=0) \|\ =\ \sum_{l\ge0}\sum_{|m|\le l} (1+ l)^{2+{1\over6}}\ |\beta_l^m(0)|<\infty\ \]
\item[(b)] $\Psi(t,x),\ \partial_t\Psi(t,x)\ \equiv0,\ t=0,\ |x|>1$,
\end{itemize}

Then,  there exists a unique solution $ \Psi(r,\Omega,t),\ \beta(\Omega,t)$, defined for $r>1,\ \Omega\in S^2$, which solves the the initial-boundary value problem. \medskip

Furthermore, define $\Gamma(\epsilon)$ to be the minimum distance of a scattering resonance (in the lower half plane) to the real axis. That is,
\begin{equation}
\Gamma(\epsilon)\ =\ \left|\ \Im\ \lambda_\star(\epsilon)\ \right|\label{Gamma-def}
\end{equation}
Then, the solution $ \Psi(x,t),\ \beta(\Omega,t)$ satisfies the \medskip

{\bf Decay Estimate:}\ 
\begin{align}
|\beta(\Omega,t)|\ &\le\ C\ \| \beta(t=0) \|\ e^{- \Gamma(\epsilon) t},\ \ \ \Omega\in S^2\label{beta-decay-est}\\
|\Psi(x,t)|\ &\le\ 
 \begin{cases} C\ \frac{1}{|x|} e^{-\ \Gamma(\epsilon)\ (t-\epsilon(|x|-1) }\ \| \beta(t=0) \| , \qquad  &1<|x|<1+\epsilon^{-1}t\\
0,\ \ &|x|>1+\epsilon^{-1}t
\end{cases}
\label{eq:Psi-decay-est}\end{align}

\end{Theo}
 \bigskip

\noindent  Our main result  is the following precise asymptotic statement about $\Gamma(\epsilon)$:
 
 \begin{Theo}\label{main-thm}
 For $0<\epsilon<\epsilon_0$, 
there exist positive constants $A_\epsilon$ and $B$, such that 
 \begin{equation}
 \Gamma(\epsilon) \ = \ \frac{A_\epsilon We }{\epsilon^3}\ 
\exp\left(-\frac{B We}{\epsilon^2}\right)\ 
 \end{equation}
where $B \approx 0.26924$ and  
 where $A_\epsilon$ has asymptotic expansion
 \begin{equation}
 A_\epsilon\ = A_0 + \mathcal{O} \left (\frac{\epsilon^2}{We} \right ),
~~{\rm where} ~~A_0 \approx e^{-2.1465} 
 \label{ABeps}
 \end{equation}
 \end{Theo}
 \bigskip
 
 \noindent{\bf Acknowledgement:}\ O.C. and S.T. were supported in part by NSF grant DMS-1108794. M.I.W. was supported in part by NSF grant DMS-10-08855.

 \section{Proof of Theorem \ref{main-thm}}
 
 We begin by deriving a characterization of the scattering resonance energies of the eigenvalue problem \eqref{eq:srp} as the zeros of an analytic function in $\mathbb{C}$. \medskip
 
 We first note that  
outgoing solutions of the three-dimensional 
Helmholtz equation are linear combinations 
of solutions of the form $h_l^{(1)}(r) Y_l^m(\Omega),\ \ |m|\le l$, where $h_l^{(1)}$ denotes the {\it outgoing} spherical Hankel function of order $l$. Thus,   we seek solutions of  \eqref{eq:srp} of the form:
\[
\Psi_{\lambda}(r,\Omega) = a\ Y_l^m(\Omega)\ h_{l}^{(1)}\left(\epsilon \lambda r\right),\ \  \beta_{\lambda}(\Omega) = b\ Y_l^m(\Omega), \ \ \ \ \ r\ge1,\ \ \  \Omega\in S^2.\]
%
where $a$ and $b$ are constants to be determined.
This choice of  $\Psi_{\lambda}$  solves the Helmholtz equation and satisfies
the outgoing radiation condition. To impose the boundary conditions at $r=1$ we 
substitute  the expressions for $\Psi_\lambda$ and $\beta_\lambda$ into \eqref{eq:srp} and obtain the
following two linear homogeneous equations for the unknown constants $a$ and $b$:
\begin{equation}
\left(\begin{array}{cc} 
\epsilon\lambda h_l^{(1)'}(\epsilon\lambda) & i\lambda \\
- i\lambda h_l^{(1)}(\epsilon\lambda) & -(l+2)(l-1)/We 
\end{array}\right)
\left(\begin{array}{c} a \\ b\end{array}\right)\ 
=\ \left(\begin{array}{c} 0 \\ 0\end{array}\right)\ ,\ l\ge2.
\end{equation}
Setting the determinant 
equal to zero yields
\[ \lambda^2\ h_l^{(1)}(\epsilon\lambda)\ +\ 
\frac{(l+2)(l-1)}{We} \epsilon\lambda\ h_l^{(1)'}(\epsilon\lambda)\ =\ 0.\]
Finally, multiplying through by $\epsilon^2$ and defining
\begin{equation}
\label{1}
 Q(l,\epsilon) = (l+2) (l-1) \frac{\epsilon^2}{We}
 \end{equation}
 yields the  following  transcendental equation 
\begin{equation}
\label{z-eqn}
z h_l^{(1)} (z) + Q (l,\epsilon) {h_l^{(1)}}^\prime (z)\ = 0  , \ \  z=\epsilon\lambda\ne0.
\end{equation} 
\medskip
Here $h_l^{(1)}$ denotes
spherical Hankel function of the first kind of order $l$.

\begin{Remark}\label{RemZ}{\rm
We will analyze 
the roots $z$ of (\ref{z-eqn}) for $We=1$ since the formula for the roots $z$ for
$We \ne 1$ are obtained from those with $We=1$ by replacing 
$\epsilon^2 $ by $\frac{\epsilon^2}{We}$.
Furthermore, since $\lambda = 
\frac{z}{\epsilon}$, proving Theorem \ref{main-thm} is equivalent to showing
that the roots of (\ref{z-eqn}) for $We =1$ satisfy the property 
\begin{equation}
\label{zresult}
\Gamma (\epsilon) =
\inf_{l \ge 2} \left \{ -\Im ~z \right \} = \frac{A_\epsilon}{\epsilon^2}
\exp \left [ -\frac{B}{\epsilon^2} \right ] 
\end{equation}
for $ 0 < \epsilon \le \epsilon_0$ sufficiently small.}
\end{Remark}

In the
following, we consider the roots in different regimes in
$l$ relative to $\epsilon$.

\begin{Definition}
  $f =O(g)$ means that for any $\epsilon$ sufficiently small, there
  exists a constant $C$ independent of $\epsilon$ so that $ |f| \le C
  |g|$. We write $f = O_s (g)$, {\it i.e. $f$ is  strictly} of order $g$, if $\Big(f =
  O(g)$ {\em and} $g=O(f)\Big)$.
  We write $f\ll g$ if $f=o(g)$, that is $f/g\to0$, as $\epsilon\to0$ and $f\gg g$ if  $|f/g|\to\infty$ as $\epsilon\to0$.
\end{Definition}
\medskip

\noindent
{\bf Outline of the proof:} In section \ref{derivation-of-a-sys} equation \eqref{z-eqn} is rewritten as a coupled set of equations for the real and imaginary parts of $z=x+iy$. In section \ref{location-of-roots} we show, via Lemma \ref{Lem1} and Lemma \ref{lem2}, that the scattering resonance of minimal imaginary part does not occur for $l=\mathcal{O}(1)$ or for $l\gg\epsilon^{-2}$. In section \ref{l-ll-eps-2}  we show that the scattering resonance of minimal imaginary part does not occur for $1\ll l\ll\epsilon^{-2}$. Therefore, the scattering resonance of interest must satisfy $l=\mathcal{O}_s(\epsilon^{-2})$. The detailed analysis of this regime and the rigorous  approximation of \eqref{Gamma-def} is in section \ref{the-region}. Section \ref{asymptotics} is an appendix containing many of the asymptotic forms of special functions.

\section{Derivation of a system of equation for $x$, $y$, where $z=x+iy$}\label{derivation-of-a-sys}

 It turns out that for large $l$ some of the roots
  $z=x+iy$ of (\ref{z-eqn})  have exponentially small $y$ (with also
  $y/x$ is exponentially small) and therefore the asymptotic
  analysis is more delicate.  We next re-express \eqref{z-eqn} as an equivalent system of equations obtained by setting its real and imaginary parts equal to zero.
We note that 
\begin{equation}
\label{2}
h_l^{(1)} (z) = j_l (z) + i y_l (z),
~~{\rm where} ~j_l (z) = \sqrt{\frac{\pi}{2 z}} J_{l+1/2} (z) 
~~,~~
y_l (z) = \sqrt{\frac{\pi}{2 z}} Y_{l+1/2} (z), 
\end{equation}
where $J_{l+1/2} (z)$ and $Y_{l+1/2} (z)$ are 
Bessel functions of order $l+1/2$. 
Substituting  into (\ref{z-eqn}) leads to
\begin{equation}
\label{3} 
(x+iy) j_l (x+iy) + i (x+iy) y_l (x+iy) + Q j_l^\prime (x+iy) + i Q y_l^\prime (x+iy) = 0     
\end{equation}
Defining the real valued functions
\begin{multline}
\label{3.1}
A_1 (x, y) = \frac{1}{2} \left [ y_l (x+iy) + y_l (x-iy) \right ]~~,   
~~A_2 (x, y) = \frac{1}{2iy} \left [ y_l (x+iy) - y_l (x-iy) \right ] \\
A_3 (x, y) = \frac{1}{2} [ y^\prime_l (x+iy) + y_l^\prime (x-iy) ]~~, 
~~A_4 (x, y) = \frac{1}{2iy} [ y_l^\prime (x+iy) - y_l^\prime (x-iy) ], 
\end{multline}
\begin{multline}
\label{3.2}
B_1 (x, y)  
= \frac{1}{2} \left [ j_l (x+iy) + j_l (x-iy) \right ]~~,   
~~B_2 (x, y) = 
\frac{1}{2iy} \left [ j_l (x+iy) - j_l (x-iy) \right ] \\
B_3 (x, y) = 
\frac{1}{2} [ j^\prime_l (x+iy) + j_l^\prime (x-iy) ]~~, 
~~B_4 (x, y)  
= \frac{1}{2iy} [ j_l^\prime (x+iy) - j_l^\prime (x-iy) ],  
\end{multline}
and separating out the imaginary and real parts of (\ref{3}), we get
\begin{equation}
\label{7}
x A_1 = - Q A_3 - y \left ( x B_2 + B_1 + Q B_4 - y A_2 \right )
\end{equation}
\begin{equation}
\label{6} 
y \left ( A_1 + x A_2 + Q A_4 \right ) = 
\left ( x B_1 + Q B_3 - y^2 B_2 \right ) 
\end{equation}
We rewrite (\ref{7}), (\ref{6}) in the form
\begin{equation}
\label{6.1}
x = - \frac{Q A_3}{A_1} - \frac{ y (x B_2 + B_1 + Q B_4 - y A_2)}{A_1} 
\end{equation} 
\begin{equation}
\label{7.1}
y = \frac{x B_1 + Q B_3 - y^2 B_2}{A_1 + x A_2 + Q A_4} 
\end{equation}

\section{Location of the roots $z$ of (\ref{z-eqn})} \label{location-of-roots}

As will be seen,  
the roots with smallest  $|\Im ~z |$ occur
for  $l=O (\epsilon^{-2} ), z=O (l)$; for them, it turns out that $\Im z$ is exponentially small in $\epsilon$. 
First, we show that roots $z$ in other regimes have larger 
$|\Im ~z|$.
\begin{Lemma}
\label{Lem1}
If $l = O(1)$ as $\epsilon \rightarrow 0$, then any root of (\ref{z-eqn}) 
in the 
sector $\arg z \in  \left ( -\pi, \pi \right )$
satisfies $z^2  \sim (l+1) Q $, and $y=\Im z = O(\epsilon^{2l+2})$ .
\end{Lemma}
\noindent See Theorem 6.2 of \cite{SW:11} for detailed asymptotics. 
\begin{proof}  
Assume first that we had $z \gg 1$. Since in this regime $\frac{{h^{(1)}}^\prime (z)}{h_l (z)} \sim i $ (see 
(\ref{14.5.1.1}) in the Appendix), 
it follows that 
$z \sim -i Q = O(\epsilon^2)$, a contradiction. Similarly,  $z=O_s (1)$
implies  $z = O(\epsilon^2)$, a contradiction again.
Thus $z \ll 1$. 
Using the well-known expansions of of $h^{(1)}_l (z),{h_l^{(1)}}' (z) $  
for small $z$ 
((\ref{A.1}) and (\ref{A.2}) in the Appendix) we get
\begin{align}
\label{0.0}
z \sim -Q \frac{ {h_l^{(1)}}^\prime (z) }{ h_l^{(1)} (z) } &= 
-Q \frac{y_l^\prime (z)}{y_l (z)} 
\left ( 1 - \frac{i j_l^\prime (z)}{y_l^\prime (z)} \right ) 
\left ( 1 - \frac{i j_l (z)}{y_l (z)} \right )^{-1}   \\
&=
\frac{(l+1) Q}{z} \left [ 1- 
\frac{z^2}{(2l-1) (l+1)} - \frac{2 z^4}{(l+1) (2l-3)(2l-3)^2 }   
+ O ({z^6}/{l^6}) \right ] \nonumber\\
&\qquad\qquad\qquad \times\ \left [ 1 + O 
\left (\frac{z}{l} \right )^{2l+1} \right ],  
\end{align}
implying by iteration 
\begin{equation}
\label{0.1}
z^2 = (l+1) Q \left [ 1 - \frac{Q}{2l-1} + 
\frac{Q^2 (l-4)}{(2l-3)(2l-1)^2}+ O \left (Q^4 l^{-4}, Q^{l+1/2} 
l^{-l-1/2} 
\right ) 
\right ],   
\end{equation}   
which is real and positive up to and including $O ( Q^l )$
Therefore, $x$  can  be computed within  errors of 
$O (Q^l) = O(\epsilon^{2l})$ 
from the series
(\ref{0.1}). 
Using (\ref{7.1}) and the small $x$ expansion of Bessel functions 
(see Appendix), the
result follows from straightforward calculations.
\end{proof}
\begin{Remark}{\rm 
\label{remz}
The expression (\ref{0.1}) holds also in the regime $1 \ll l  \ll \epsilon^{-2}$ 
}
\end{Remark}
\begin{Lemma} 
\label{lem2}
If $l  \gg \epsilon^{-2} \gg 1$, then any root of (\ref{z-eqn}) in the 
sector $\arg z \in (-\pi, \pi )$ satisfies  $z \sim ~-i Q $, 
implying $-\Im z = O (l^2 \epsilon^2) \gg \epsilon^{-2} $. 
\end{Lemma}
\begin{proof}
 We first claim that there are no roots $z \ll l$. Indeed, otherwise the expansion of
  $h_l^{(1)} (z) $ and ${h_l^{(1)}}^\prime (z) $ in this regime (see
  (\ref{A.1}), (\ref{A.2}) in the Appendix) would lead to a
  contradiction: $z=O (\sqrt{lQ}) = O (l^{3/2} \epsilon ) \gg l $.  On
  the other hand, the existence of roots $z = O_s (l)$ implies, using
  the asymptotics of $h_l^{(1)}$ and $(h_l^{(1)})'$ (see \S 7.3- \S
  7.5 in the Appendix for results in different ranges of
  $\frac{z}{\sqrt{l (l+1)}}$) in (\ref{z-eqn}),  a
  contradiction: $\frac{z}{\sqrt{l (l+1)}} = O_s \left ( \frac{Q}{l}
  \right ) = O_s (l \epsilon^{2} ) \gg 1$.  Thus, $|z| \gg l$. Using
  the asymptotic behavior of $h^{(1)}_l (z)$ and its derivative in
  this regime leads (see section \ref{7.2}) to $z \sim -i Q$.
\end{proof}

\section{Analysis of 
the case $ 1\ll l\ll \epsilon^{-2} $}\label{l-ll-eps-2}

\begin{Lemma}
\label{lem3}
If $1 \ll l \ll \epsilon^{-2}$, then any root of (\ref{z-eqn}) 
is given by (\ref{0.1}), implying that
$z = O_s (l^{3/2} \epsilon) \ll l \ll \epsilon^{-2}$.  
\end{Lemma}
\begin{proof}
Using the asymptotics  \eqref{14.5.1.1}  of $h^{(1)}_l (z)$ 
and $(h^{(1)}_l (z))'$ 
in (\ref{z-eqn}), the assumption $z \gg l$  leads to a contradiction:
$z \sim - i Q = O(l^2 \epsilon^2 ) \ll l$.
Assuming  $z=O_s (l)$ and using  the asymptotics
of $h_l^{(1)}$ and $(h_l^{(1)})'$  (Appendix,  \S 7.3- \S 7.5)
in 
(\ref{z-eqn}), also ends up in a contradiction:
$\frac{z}{\sqrt{l (l+1)}} = O_s \left ( \frac{Q}{l} \right ) = 
O_s (l \epsilon^2 ) \ll 1$.
In the regime $z=o(l)$, 
exploiting the asymptotics of $h_l^{(1)} (z)$, we obtain (\ref{0.0}) and
therefore (\ref{0.1}).
\end{proof}

By Lemma \ref{lem3}, any root of (\ref{z-eqn}) in this case satisfies $z
=o(l)$.  From (\ref{0.1}) (see Remark \ref{remz}), it follows that
the roots are close to the real axis. 
In this section, we estimate $y = \Im ~z$ in this regime.

\begin{Lemma}
\label{lem4}
Any root $z=x+iy$ of (\ref{z-eqn}) in the regime $1 \ll l \ll \epsilon^{-2}$ 
satisfies
\begin{equation} 
\label{1.2} 
y 
\sim -\frac{\sqrt{Q (l+1)}}{2e} 
\left ( \frac{e Q^{1/2} (l+1)^{1/2}}{2l} \right )^{2l+1} 
\left [ 1 - \frac{Q}{(2l-1)} + O (Q^2/l^2) \right ]^{l+1/2}
\end{equation}
\end{Lemma}
\begin{proof}
Since $z \ll l$, we use
(\ref{0.1}) and the expansions of Bessel functions in this regime (see Appendix) to conclude
\begin{multline}
\label{1.3}
\frac{j_l (z)}{y_l(z)} = -\frac{1}{2e} 
\left ( \frac{ze}{2l} \right )^{2l+1} \left [ 1 + O \left(\frac{z^2}{l^2}\right) 
\right ] \\
=-\frac{1}{2e} \left( \frac{e Q^{1/2} (l+1)^{1/2}}{2l} \right)^{2l+1} 
\left [ 1 - \frac{Q}{(2l-1)} + O (Q^2/l^2) \right ]^{l+1/2}
\end{multline}
In the same way,
\begin{equation}
\label{1.3.1}
1+ \frac{Q B_3}{x B_1} = 2 + O
\left (\frac{1}{l}, \frac{Q}{l} \right ) ~~,~~  
1 + x \frac{A_2}{A_1} + \frac{Q A_3}{A_1} = 2 + O (Q/l) 
\end{equation}
Using (\ref{7.1}), one obtains (\ref{1.2}).
\end{proof} 
\begin{Remark}{\rm Since $Q^{1/2}/l^{1/2} $ scales like $l^{1/2} \epsilon \ll 1$,
we have $|y| \sim C l \left ( \epsilon l^{1/2} \right )^{2l+2} $, which is a decreasing function of $l$ in the
regime $1 \ll l \ll \epsilon^{-2}$. Therefore, $|y|$ is minimal at the right end, and this minimal value turns out to be larger than the values obtained in the next section.}
\end{Remark}

\section{Analysis of the case $l = O_s (\epsilon^{-2})$}\label{the-region}
From (\ref{1}), it follows that when $l=O_s (\epsilon^{-2} )$, $Q=O_s(\epsilon^{-2} )$.
If  $z=O_s(l)$, but $|z-l| \gg l^{1/3} $, from the
asymptotics  (\ref{14.5.6}) and (\ref{14.5.7}), it follows that
(\ref{z-eqn}) results in
\begin{equation}
\label{8.0}
z^2 \sim -i Q \sqrt{\frac{z^2}{l(l+1)} -1 }
\end{equation}
Hence  $|y|=|\Im z| \gg 1 $. In the
crossover regime, $1-\frac{z}{\sqrt{l(l+1)}} = O(l^{-2/3})$, the asymptotics of Bessel functions involve
the Airy functions  Ai and Bi  (\ref{16.5}) and (\ref{16.6})) whose ratio is $O_s(1)$,  and thus
$y$ cannot be exponentially small.

So, we restrict to the sub-region
$\Big | \frac{z}{\sqrt{l(l+1}} 
\Big | < 1$ and $1-\frac{z}{\sqrt{l(l+1)}} \gg l^{-2/3}$.  
In this regime, since $\frac{j_l (z)}{y_l (z)}$ is
exponentially small in $l$, by
(\ref{7.1}), any root $z$  must have 
$y$  exponentially small.  
We introduce the  scaled variables
\begin{equation}
\label{8}
\sqrt{l(l+1)}
= l_* \epsilon^{-2},~~ Q = l_* \epsilon^{-2} q ~~, 
~x = \frac{l_*}{\epsilon^2} \zeta ~~,~~~
-\epsilon^2 \log \left [ -\epsilon^2 y  \right ] = \eta 
\end{equation} 
From the definition of $Q$, it follows that (for $We =1$),
\begin{equation}
\label{8.0}
q=l_*- \frac{2 \epsilon^4}{l_*}
\end{equation}
From  (\ref{6.1}) and (\ref{7.1}),  we can replace
(\ref{7.1}) by
\begin{multline}
\label{8.0.0}
y = \frac{x B_1 + Q B_3 - y^2 B_2}{D}
~~{\rm where} ~~ \\
D = A_1 + x A_2 + Q A_4 + \sqrt{ \frac{l(l+1)}{x^2} - 1}
\left [ x A_1 + Q A_3 + y
\left ( x B_2 + B_1 + Q B_4 - y A_2 \right )
\right ]\ .
\end{multline}
Note: the square-bracketed quantity on the previous line vanishes by \eqref{7.1}.
With the notation in (\ref{8}),
(\ref{6.1}) and (\ref{8.0.0}) imply
\begin{equation}
\label{6.2}
0=\zeta + \frac{q A_3}{A_1} - 
\frac{e^{-\eta/\epsilon^2}}{\epsilon^2 A_1}   
\left [ \frac{B_1}{l_*\epsilon^2} +\zeta B_2 + q B_4
+ \frac{A_2}{\epsilon^4 l_*}   
e^{-\eta/\epsilon^2} \right ] =: F(\zeta, \eta; \epsilon)   
\end{equation}
\begin{equation}
\label{7.2}
0=\eta + \epsilon^2 \log \left [ -\epsilon^2 \frac{(x B_1 + Q B_3
- y^2 B_2)}{D} \right ]
=: G(\zeta, \eta; \epsilon) 
\end{equation}
\begin{Lemma}
\label{LemExist}
{\rm For any fixed $\delta \in (0, 1)$, 
if $l_* \in \left ( \frac{\delta^2}{\sqrt{1-\delta^2} } ,
\frac{(1-\delta)^2}{\sqrt{1-(1-\delta)^2}} \right )$, 
$\zeta\in (\delta,1-\delta)$ and $\eta>\eta_0=2l_*\int_{1-\delta}^1 t^{-1}\sqrt{1-t^2}dt$, then there exists
$\epsilon_0$ such that for $|\epsilon| \le \epsilon_0$,
$F$ and $G$ are smooth functions of $\zeta$, $\eta$, $l_*$ and $\epsilon$, with
\begin{equation}
\label{6.4}
F(\zeta, \eta; 0) = 
\zeta-l_* \frac{\sqrt{1-\zeta^2}}{\zeta} ~~~~~~~~,~~~~~~~ 
~~G(\zeta, \eta; 0) = \eta
-2 l_* \int_\zeta^1 \frac{\sqrt{1-t^2}}{t} dt  
\end{equation}
and 
${\rm det} 
\frac{\partial\{F, G\}}{\partial\{\zeta,\eta\}} \Big|_{\epsilon=0}\ne 0 $
Furthermore, in these intervals, 
the system of
equations $\{F(\zeta, \eta; \epsilon)=0,G(\zeta, \eta;\epsilon)=0 \}$
has a unique smooth solution $\left ( \zeta (l_*, \epsilon), \eta (l_*, \epsilon) \right )$, 
for sufficiently small $\epsilon$, with asymptotic behavior in $\epsilon$ 
determined implicitly from
\begin{equation}
\label{e2.0.zeta}
l_* = \frac{\zeta^2}{\sqrt{1-\zeta^2}} -\frac{\epsilon^2 (1-2 \zeta^2)}{2 (1-\zeta^2)^{3/2}} 
+ O(\epsilon^4), 
\end{equation}
\begin{multline}
\label{e2.0}
\eta = 2 l_* \int_\zeta^1 (t^{-2} - 1)^{1/2} dt
- \epsilon^2 \log \left [ l_* \zeta \right ]
\\
+\epsilon^2 \log \left [ 2 +
\frac{1}{2 (1-\zeta^2)}
- \frac{l_*^2 (1-2\zeta^2)}{2 \zeta^4} \right ]
+ O (\epsilon^4)
\end{multline}}
\end{Lemma}
\begin{proof} 
Using a standard integral representation of $j_l (\sqrt{l(l+1)} \zeta)$ 
and $y_l (\sqrt{l(l+1)} \zeta)$, (\ref{14.5.1}) and (\ref{14.5.2}),  and noting
that $\sqrt{l(l+1)} = l_* \epsilon^{-2}$, smoothness of
$F$ and
$G$ in $\epsilon$ follows provided
 $A_1$ and $D$ 
in (\ref{6.2}) and (\ref{7.2})
are nonzero,
which is the case, as shall be seen, for $\epsilon=0$.
Routine calculations, using the asymptotics 
(\ref{15})-(\ref{16.4}) 
show that
\begin{equation}
\label{eqzetaq}
\zeta+q \frac{A_3}{A_1} =
\zeta+q \frac{y_l^\prime(x)}{y_l(x)} + O (e^{-\eta_0 \epsilon^{-2}} )
= \zeta-\frac{l_* \sqrt{1-\zeta^2}}{\zeta} \left [
1+ \epsilon^2 \frac{1-2 \zeta^2}{2 l_* (1-\zeta^2)^{3/2}}
+ O (\epsilon^4) \right ],
\end{equation}
Therefore,
\begin{equation}
\label{eqFeps}
F(\zeta, \eta; \epsilon) = \zeta - l_* \frac{\sqrt{1-\zeta^2}}{\zeta}
-\frac{\epsilon^2 (1-2\zeta^2)}{2 \zeta (1-\zeta^2)}
+ O (\epsilon^4), 
\end{equation}
from which the formula for $F(\zeta, \eta; 0)$ in (\ref{6.4}) 
follows. Now, setting $F=0$ in (\ref{eqFeps}) immediately leads to 
(\ref{e2.0.zeta}).
For the second part, from the asymptotics of $j_l$, $y_l$ in this regime, (\ref{15})-(\ref{16.4}),
we get
\begin{multline}
x B_1 + Q B_3 -y^2 B_2 = x j_l (x) \left [ 1
+ \frac{q j_l^\prime (x)}{\zeta j_l (x)} + O(y^2) \right ]
\\
= x j_l (x)
\left [ 1 + \frac{l_* \sqrt{ 1-\zeta^2}}{\zeta^2}
- \frac{\epsilon^2 (1-2 \zeta^2)}{2 \zeta^2 (1-\zeta^2)} + O (\epsilon^4)
\right ]
\end{multline}
\begin{multline}
D =
A_1 + x A_2 + Q A_4 +
\sqrt{\frac{l (l+1)}{x^2} -1 } \left [ x A_1 + Q A_3 +
y \left ( x B_2 + B_1 + Q B_4 - y A_2 \right ) \right ] \\
= y_l \left \{ \frac{1}{2 (1-\zeta^2)} + \frac{2 l_*}{\zeta^2}
\sqrt{1-\zeta^2} -
\frac{l_* (1-2\zeta^2)}{2 \zeta^2 \sqrt{1-\zeta^2}} + O(\epsilon^2) \right \}
\end{multline}
Therefore, it follows that
\begin{multline}\label{m1}
G(\zeta, \eta; \epsilon) = \eta +\epsilon^{2} \log \left \{ -\epsilon^2
\frac{x B_1 + Q B_3 -y^2 B_3}{D} \right \}
= \eta+\epsilon^2 \log \left [-\frac{l_* \zeta j_l (x)}{y_l (x)} \right ] \\
+ \epsilon^2 \log \left [ 1 + \frac{l_* \sqrt{1-\zeta^2}}{\zeta^2} \right ]
-\epsilon^2 \log \left [
\frac{1}{2 (1-\zeta^2)} + \frac{2l_*}{\zeta^2} \sqrt{1-\zeta^2}
- \frac{l_* (1-2\zeta^2)}{2 \zeta^2 \sqrt{1-\zeta^2} } \right ]
+ O (\epsilon^4)
\end{multline}
From equation (\ref{16.4}) in the Appendix, we obtain
\begin{equation}\label{e2}
-2 \frac{j_l (x)}{y_l (x)} = \exp \left [ -2 l_* \epsilon^{-2} \int_\zeta^1
(t^{-2} - 1)^{1/2} dt \right ] ,
\end{equation}
\begin{multline}\label{m1.1}
G(\zeta, \eta; \epsilon) = \eta -2 l_* \int_\zeta^1 (t^{-2} - 1)^{1/2} dt
+\epsilon^{2} \log \left ( \frac{1}{2} l_* \zeta
\left [ 1 + l_* \frac{\sqrt{1-\zeta^2}}{\zeta^2} \right ] \right )
\\
-\epsilon^2 \log \left [
\frac{1}{2 (1-\zeta^2)} + \frac{2l_*}{\zeta^2} \sqrt{1-\zeta^2}
- \frac{l_*^2 (1-2\zeta^2)}{2 \zeta^4}
\left ( \frac{\zeta^2}{l_* \sqrt{1-\zeta^2} } \right )
\right ]
+ O (\epsilon^4)
\end{multline}
from which the expression of $G(\zeta, \eta; 0)$ in (\ref{6.4}) follows.
The Jacobian of $F(\zeta, \eta; 0)$, $G(\zeta, \eta; 0)$ 
with respect to $(\zeta, \eta)$ is clearly nonzero. The
statement of existence and uniqueness of $\left ( \zeta (l_*, \epsilon), \eta (l_*, \epsilon) \right )$ 
now follows from the implicit function theorem.
Furthermore, using (\ref{m1.1}) and (\ref{e2.0.zeta}) in $F=0$, $G=0$ immediately implies 
(\ref{e2.0}).
\end{proof}

\begin{Note}{\rm We now comment on why the interval restrictions in Lemma \ref{LemExist}
for $l_*$, $\zeta$ and $\eta$ do not matter for
finding the smallest $|y|$ as a function of $l_*$.
First, if 
$l_* < \frac{\delta^2}{\sqrt{1-\delta^2}}$, then
for sufficiently small $\delta$, 
the conclusions of Lemma \ref{lem4} hold and the corresponding
$|y|$ is not as small as implied by Lemma \ref{LemExist} for $l_*$ in the given interval.
On the other hand if $l_* > 
\frac{(1-\delta)^2}{\sqrt{1-(1-\delta)^2}} $ for sufficiently small $\delta$,
then the conclusions of Lemma \ref{lem2} hold where $y = \Im ~z$ is no longer small.  
Thus, the restriction 
$l_* \in \left (\frac{\delta^2}{\sqrt{1-\delta^2}} ~,~ 
\frac{(1-\delta)^2}{\sqrt{1-(1-\delta)^2}} \right ) $ is appropriate.
Furthermore, for $l_*$ in this
interval, there is no need to consider the possibility $\zeta \notin \left (\delta, 1-\delta \right )$ 
since we have already shown that $l=O_s \left ( \epsilon^{-2} \right )$, implies
$x = O_s \left ( \epsilon^{-2} \right )$ and argued at the outset that we only
need to consider $1-\zeta \gg l^{-2/3}$. When $(1-\zeta)$ is small,
the asymptotic behavior
of $\frac{j_l (x)}{y_l (x)}$ is given by a ratio of Airy functions, as is the case for
$1-\zeta = O(l^{-2/3})$ (see (\ref{16.5}) and (\ref{16.6})) and the
resulting $|y|$ is not as small as obtained in \eqref{eqym} below, as a consequence of  Lemma \ref{LemExist}.
Also, we need not worry about
possible solutions for which $\eta < \eta_0$, since we are seeking
to maximize $\eta$ (minimizing $|y|$).}
\end{Note}

\subsection{Maximizing $\eta$ as a function of $l_*$ and determination of $\Gamma(\epsilon)$ of \eqref{Gamma-def}}

Since $\Im ~z =: y = -\epsilon^{-2} \exp \left [ -\frac{\eta}{\epsilon^2} \right ]$, 
minimizing $|y|$ as a function of $l_*$ corresponds to maximizing $\eta$.
Now, 
(\ref{e2.0.zeta}) implies that
$l_*$ increases 
with $\zeta $ in any compact subset of $(\delta, 1-\delta)$ for all sufficiently small $\epsilon$.
For now, we consider
$l_*$ as a continuous variable for the purposes of
finding the maximum value of $\eta$. We will show later that 
the maximal value of $\eta$ to the order calculated is the same
if  $l_* $ takes  discrete values: $\sqrt{l (l+1)} \epsilon^2$ for
$l \in \mathbb{N}$.

We seek a critical value 
$l_*$ for which
$\partial_{l_*} \eta = 0$ and
$\partial_{l_* l_*}^2 \eta < 0$ implying a maximum of $\eta$.
We note that
\begin{multline}
\label{14.6}
0=\partial_{l_*} \eta=
2 \int_\zeta^1
(t^{-2} - 1)^{1/2} dt
-\frac{2 l_*}{\zeta} \zeta_{l_*} (1-\zeta^{2})^{1/2}
- \frac{\epsilon^2}{l_*} - \frac{\epsilon^2 \zeta_{l_*}}{\zeta}
\\
+ \epsilon^2
\zeta_{l_*} \partial_\zeta \log \left [ 2 +
\frac{1}{2 (1-\zeta^2)}
- \frac{l_*^2 (1-2\zeta^2)}{2 \zeta^4} \right ]
+ \epsilon^2
\partial_{l_*} \log \left [ 2 +
\frac{1}{2 (1-\zeta^2)}
- \frac{l_*^2 (1-2\zeta^2)}{2 \zeta^4} \right ]
+ O (\epsilon^4)
\end{multline}
We also note that (\ref{e2.0.zeta}) implies
\begin{equation}
\label{eqzetal}
\zeta_{l_*} = \frac{(1-\zeta^2)^{3/2}}{\zeta
(2-\zeta^2)} -
\frac{(1+2 \zeta^2)\sqrt{1-\zeta^2}}{ 2 \zeta ( 2-\zeta^2)^2
} \epsilon^2 + O (\epsilon^4)
\end{equation}
Substituting (\ref{e2.0.zeta}) and (\ref{eqzetal}) into  
(\ref{14.6}),
we obtain
\begin{equation}
\label{14.9.0}
0 = \frac{1}{2} \eta_{l_*} = \int_{\zeta}^1 \sqrt{t^{-2} - 1} ~dt -
\frac{(1-\zeta^2)^{3/2}}{(2-\zeta^2)}
+ \epsilon^2
\frac{\sqrt{1-\zeta^2} (2 \zeta^4+5 \zeta^2 - 4)}{
2 \zeta^2 (2-\zeta^2)^2}
+ O(\epsilon^4)
=: g(\zeta; \epsilon)
\end{equation}
It is clear that the solution $\zeta$
of $g(\zeta; \epsilon)=0$ has the behavior
\begin{equation}
\label{14.9.0.1}
\zeta = \zeta_{m,0} + \epsilon^2 \zeta_{m,2} + O(\epsilon^4)
\end{equation}
where
\begin{equation}
\label{eqzeta0}
0 = \int_{\zeta_{m,0}}^1 \sqrt{t^{-2} - 1} dt -
\frac{(1-\zeta_{m,0}^2)^{3/2}}{(2-\zeta_{m,0}^2)}
= g (\zeta_{m,0}; 0) =: g_0 (\zeta_{m,0}).
\end{equation}
Since
\begin{equation}
\label{eqg0p}
g_0^\prime (\zeta)
= 2 \sqrt{1-\zeta^2} \left (\frac{2-(2-\zeta^2)^2}{\zeta
(2-\zeta^2)^2 } \right ),
\end{equation}
it follows that 
\begin{equation}
\label{eqzeta2}
\zeta_{m,2} =
-\frac{\sqrt{1-\zeta_{m,0}^2} (2 \zeta_{m,0}^4+5 \zeta_{m,0}^2 - 4)}{
2 g_0^\prime (\zeta_{m,0}) \zeta_{m,0}^2 (2-\zeta_{m,0}^2)^2}
=-\left \{ \frac{(2 \zeta_{m,0}^4+5 \zeta_{m,0}^2 - 4)}{
4 \zeta_{m,0}
\left (2 - (2-\zeta_{m,0}^2)^2 \right )}
\right \}
\end{equation}
We now seek to find $\zeta_{m,0}$ which corresponds to the maximal $\eta$.
We note from (\ref{eqg0p}) that 
$g_0^\prime < 0$ for
$\zeta \in \left (0, \sqrt{2-\sqrt{2}} \right ) =: J$
and $g_0^\prime > 0$  for $\zeta \in
\left ( \sqrt{2-\sqrt{2}}, 1 \right ) $.
Since we are seeking a maximum for $\eta$, we must have $g_0^\prime < 0 $ at the corresponding $\zeta$ (noting that
$\zeta_{l_*} > 0$). 
Thus the roots of 
$g_0 (\zeta)=0$ only need to be sought for $\zeta \in J$.
At the right end of the interval $J$,
explicit integration --in terms of elementary functions-- gives
$g_0 < 0$ ($\approx -0.0678$), while
$g_0 \rightarrow +\infty$ when $\zeta \rightarrow 0$. Since
$g_0$ is monotonic in $J$, there exists unique $\zeta_{m,0}$ satisfying
$g_0 (\zeta_{m,0}) = 0$.
The explicit calculation of $g_0$ gives
$g_0 (0.58) > 0$ and $g_0 (0.59) < 0 $, implying
that the unique maximum $\zeta_{m,0}$ is in $(0.58, 0.59)$--in fact
$\zeta_{m,0} = 0.58134..$. Using (\ref{eqzeta2}) we get  $\zeta_{m,2} \approx -1.1743..$.
Equation (\ref{e2.0.zeta})
also implies that the critical $l_*$ ($=l_{*,m}$) that maximizes $\eta$ is given by
\begin{multline}
l_{*,m}=\frac{\zeta_{m,0}^2}{\sqrt{1-\zeta{m,0}^2}}
+ \epsilon^2 \left [
\frac{\zeta_{m,1} \zeta_{m,0} (2-\zeta_{m,0}^2)}{
(1-\zeta_{m,0}^2)^{3/2} }
- \frac{(1-2 \zeta_{m,0}^2)}{2 (1-\zeta_{m,0}^2)^{3/2}}
\right ] + O(\epsilon^4) \\
=: l_{m,0} + \epsilon^2 l_{m,2} + O(\epsilon^4)
\approx 0.41535 - 2.4071 \epsilon^2
+ O(\epsilon^4) 
\end{multline}
Therefore, it follows that the maximized $\eta = \eta_m$ satisfies
\begin{multline}
\label{eqetam}
\eta_m = 2 l_{m,0} \int_{\zeta_{m,0}}^1 (t^{-2} - 1)^{1/2} dt
+ \epsilon^2 \left \{ 2 l_{m,2}
\int_{\zeta_{m,0}}^1 (t^{-2} - 1)^{1/2} dt
- 2 l_{m,0} \zeta_{m,2} (\zeta_{m,0}^{-2} - 1)^{1/2} \right. \\
\left. - \log [l_{m,0} \zeta_{m,0}] 
+ \log \left [ 2 +
\frac{1}{2 (1-\zeta_{m,0}^2 )}-
\frac{l_{m,0}^2 (1-2 \zeta_{m,0}^2 )}{2\zeta_{m,0}^4} \right ]
\right \} =:\eta_{m,0}+\epsilon^2 \eta_{m,2} + O (\epsilon^4) \\
\approx 0.26924 + 2.1465 \epsilon^2 + O(\epsilon^4)\ .
\end{multline}
Therefore the minimum value of $|y|$ is
\begin{multline}
\label{eqym}
|y_m| = \epsilon^{-2}
\exp \left [ - \epsilon^{-2} \eta_{m,0} - \eta_{m,2} \right ]
\left [ 1 + O(\epsilon^2 \right ]
\\\approx -\epsilon^{-2} \exp
\left [ -0.26924 \epsilon^{-2} - 2.1465 \right ]
\left [ 1 + O(\epsilon^2 \right ]
\end{multline}
The proof of Theorem \ref{main-thm} now follows from Remark \ref{RemZ}.
\begin{Remark}{\rm
The calculation (\ref{eqetam}) assumed $l_*$ to be a continuous variable rather
than  discrete,  $l_* = \sqrt{ l (l+1)} \epsilon^2$ with $ l \in \mathbb{N}$.
This however makes
no difference to the leading order result in (\ref{eqym}). Maximization of $\eta$ over
$l \in \mathbb{N}$ results in an optimal $l_{*}$ that is different from the 
computed $l_{*,m}$ by $O(\epsilon^2)$. Since $\eta$ is a smooth function of
$l_*$, for $l_* - l_{*,m} = O(\epsilon^2)$,
$\eta (l_*) - \eta (l_{*,m}) = O \left (l_*-l_{*,m} \right )^2  = O (\epsilon^4)$ and therefore
discreteness does not affect 
the leading order asymptotic result (\ref{eqym}).}  
\end{Remark}

\section{Appendix: Asymptotics of $j_l$ and $y_l$ in different regimes}\label{asymptotics}

The results quoted below are standard and either given in standard references 
such as \cite{Abramowitz}, 
\cite{Olver} or follow directly from them. 

The spherical Bessel functions $j_l(z)=\sqrt{\frac{\pi}{2z}} J_{l+1/2} (z)$ 
and $y_l (z) = \sqrt{\frac{\pi}{2z}} Y_{l+1/2} (z) $ 
have the following integral representations
which immediately follows from equations (10.9.6),
(10.9.7) in \cite{Abramowitz} (see also http://dlmf.nist.gov/10.9) 
using $\nu = l+1/2$
\begin{multline}
\label{14.5.1}
j_l (z) = \sqrt{\frac{1}{2 \pi z}} 
\Big \{  \int_0^{\pi} \cos \left [ (l+\frac{1}{2} ) \tau
- z \sin \tau \right ] d \tau  \\+ (-1)^{l+1} 
\int_0^\infty \exp \left [ - z \sinh t - (l+\frac{1}{2} ) t \right ] dt 
\Big \}
\end{multline}
\begin{multline}
\label{14.5.2}
y_l (z) = \sqrt{\frac{1}{2 \pi z}} 
\Big[ (-1)^{l+1}
\int_0^\pi \cos \left [ (l+\frac{1}{2} ) \tau + z \sin \tau \right ] d\tau 
\\-
\int_0^\infty \exp \left [ - z \sinh t + (l+\frac{1}{2} ) t \right ] dt 
\Big]
\end{multline}
which results in the asymptotic representations 
below by
standard Laplace method for asymptotics of integrals \cite{Olver}. 
The Bessel functions $j_l (z)$, $y_l (z)$ and the 
Hankel
function $h_l^{(1)} (z) = j_l (z) + i y_l (z)$ satisfy
\begin{equation}
\label{14.5}
u^{\prime \prime} + \frac{2}{z} u^\prime + \left ( 1 -
\frac{l(l+1)}{z^2} \right ) u = 0 
\end{equation}

\subsection{The regime $z \ll l$:}\label{7.1}
We only present the asymptotic regimes relevant to this analysis.
For $z \ll l$ we have (\
see equations (10.1.2) and (10.1.3), page 437 of \cite{Abramowitz}\ )
\begin{equation}
\label{A.1} 
j_l (z) = \frac{z^{l} 2^l l!}{(2l+1)!} 
\left [ 1 + O \left (\frac{z^2}{(l+1)^2}
\right )
\right ]  
\end{equation}
\begin{equation}
\label{A.2} 
y_l (z) = -\frac{(2l-1)!}{2^{l-1} (l-1)! z^{l+1}}        
\left [ 1 +
\frac{z^2}{2(2l-1)} + \frac{z^4}{8 (2l-1) (2l-3)} 
+ O \left (\frac{z^6}{(l+1)^3} \right ) 
\right ] 
\end{equation}
(Three orders are indeed needed because of cancellations.)
For $l\gg z \gg 1$, using 
$\Gamma (x+1) \sim \sqrt{2\pi} e^{-x} x^{x+1/2}$
for large $x$ we note the simplification:
\begin{equation}
\label{A.3} 
j_l (z) \sim  
\frac{2^{-3/2}}{l}
\left ( \frac{z e}{2 l} \right )^{l}  
\left [ 1 + O \left (\frac{z^2}{l^2} \right ) \right ] 
\end{equation}
\begin{equation}
\label{A.4} 
y_l (z) \sim -\frac{e}{l\sqrt{2}} 
\left ( \frac{z e}{2 l} \right )^{-l-1}  
\left [ 1 + \frac{z^2}{2 (2l-1)} 
+\frac{z^2}{2(2l-1)} + \frac{z^4}{8(2l-1) (2l-3)} 
+ O \left (\frac{z^6}{l^3} \right ) \right ] 
\end{equation}
Since the asymptotics is differentiable \cite{Wasow}, 
this implies
\begin{equation}
\label{A.5}
\frac{j_l^{\prime }(z)}{j_l (z)} 
= \frac{l}{z} \left [ 1 
+O \left (\frac{z^2}{l^2} \right ) \right ]
\end{equation}
\begin{equation}
\label{A.6}
\frac{z y_l^{\prime }(z)}{y_l (z)} 
= - (l+1) + \frac{z^2}{2l-1} 
+ \frac{z^4}{(2l-3)(2l-1)^2} 
+ O \left(\frac{z^6}{l^5}\right) 
\end{equation}
\begin{equation}
\label{A.4} 
\frac{y_l^{\prime \prime} (z)}{y_l (z)}
= (l+1) (l+2) z^{-2} - \frac{2l+1}{2l-1} 
-\frac{2z^2}{(2l-3)(2l-1)^2}  
+ O \left ( \frac{z^4}{l^4} \right ) 
\end{equation} 
\begin{equation}
\label{A.7}
\frac{j_l (z)}{y_l(z)} = -\frac{1}{2e} 
\left ( \frac{ze}{2l} \right )^{2l+1} \left [ 1 + O (\frac{z^2}{l^2}) 
\right ]
\end{equation} 

\subsection{The regime $z \gg l$,  $\arg z \in (-\pi, \pi)$} \label{7.2}
In this case
\begin{equation}
\label{14.5.1.1}
h_l^{(1)} (z)  \sim \frac{1}{z} \exp 
\left [ i \left (z - \frac{(l+1)}{2} \pi \right ) \right ]  
\end{equation}
\begin{equation}
\label{14.5.2.1}
j_l (z)  \sim \frac{1}{z} 
\cos \left (z - \frac{(l+1)}{2} \pi \right ) 
\end{equation}
\begin{equation}
\label{14.5.3}
y_l (z)  \sim \frac{1}{z} 
\sin \left (z - \frac{(l+1)}{2} \pi \right ) 
\end{equation}

\subsection{The regime $l \gg 1$, $\xi = \frac{z}{\sqrt{l (l+1)}} =O_s(l)$, 
$1-|\xi| > l^{-2/3}$, $|y| \ll l^{1/3}$, 
$\arg (1-\xi) \in (-\pi/3, \pi/3) $}  \label{7.3}
From (\ref{14.5.1}) and (\ref{14.5.2}), by the
saddle point method 
(\cite{Olver}),
we obtain
\begin{multline}
\label{15}
j_l (z) = \frac{1}{2 \sqrt{l (l+1)} \xi} 
\left ( \xi^{-2}  - 1 \right )^{-1/4}   \\ 
\times \exp \left [ - l^{1/2}(l+1)^{1/2} \int_{\xi}^1 
\left( t^{-2} - 1 \right )^{1/2} dt \right ]  
\left [1 + \frac{u(\xi)}{l} 
+ O(l^{-2} )\right ]   
\end{multline}
where
\begin{equation}
\label{16} 
u(\xi) = -\frac{5}{24 (1-\xi^2)^{3/2}} + \frac{1}{8 (1-\xi^2)^{1/2}}
+ \frac{1}{16} \log \left (
\frac{1-\sqrt{1-\xi^2}}{1+\sqrt{1-\xi^2}} 
\right )
\end{equation}
\begin{multline}
\label{16.1}
y_l (z) \sim -\frac{1}{\sqrt{l(l+1)} \xi} 
\left ( \xi^{-2} - 1 \right )^{-1/4}  \\
\times    
\exp \left [ l^{1/2}(l+1)^{1/2} \int_{\xi}^1 
\left[ t^{-2} - 1 \right ]^{1/2} dt \right ]   
\left [1 - \frac{u(\xi)}{l} 
+ O(l^{-2}) \right ]   
\end{multline}
\begin{equation}
\label{16.2}
\frac{y_l^\prime (z)}{y_l (z)} = 
-\sqrt{\xi^{-2} -1} 
- 
\frac{1-2 \xi^2}{2 l \xi (1-\xi^2)}   + O(l^{-2} )  
\end{equation}
\begin{equation}
\label{16.3}
\frac{j_l^\prime (z)}{j_l (z)} = 
\sqrt{\xi^{-2} -1} - 
\frac{1-2 \xi^2}{2 l \xi (1-\xi^2)}   + O(l^{-2} )  
\end{equation}
In this regime, we have
\begin{equation}
\label{16.4} 
-\frac{2 j_l (z)}{y_l (z)} = 
\exp \left [ - 2\sqrt{l(l+1)} \int_\xi^1
(t^{-2} - 1)^{1/2} dt \right ] \left [1 
+ \frac{2}{l} u(\xi) + O (l^{-2} ) \right ]  
\end{equation}

\subsection{The regime $l \gg 1$, $\xi = \frac{z}{\sqrt{l (l+1)}} =O_s(l)$,
$|\xi| - 1 \gg l^{-2/3}$, $|y| \ll 1$,  
$\arg (\xi-1) \in (-\pi/3, \pi/3) $} \label{7.4} 
In this case  we have
\begin{equation}
\label{14.5.6} 
j_l (z) \sim \frac{1}{\sqrt{l (l+1)} \xi} (1 -\xi^{-2} )^{-1/4}  
\cos \left [ \sqrt{l(l+1)} \int_1^\xi \sqrt{(1-t^{-2} )} 
dt - \frac{(l+1) \pi}{2} \right ] 
\end{equation} 
\begin{equation}
\label{14.5.7} 
y_l (z) \sim \frac{1}{\sqrt{l(l+1)} \xi} (1 -\xi^{-2} )^{-1/4}  
\sin \left [ \sqrt{l(l+1)} \int_1^\xi \sqrt{(1-t^{-2} )} dt 
- \frac{(l+1) \pi}{2} 
\right ] 
\end{equation} 

\subsection{ The regime  $ l \gg 1$, $\xi=\sqrt{z}{\sqrt{l(l+1)}} = O_s (l)$,
$1-\xi = O(l^{-2/3})$}\label{7.5}
Here   the asymptotic behavior is  given by
\begin{equation}
\label{16.5}
j_l (z) \sim \frac{\sqrt{\pi} l^{1/6}}{2^{1/6} z} {\rm Ai} \left (l^{2/3} 2^{1/3} 
[1-\xi] \right )    
\end{equation}
\begin{equation}
\label{16.6}
y_l (z) \sim -\frac{\sqrt{\pi} l^{1/6}}{2^{1/6} z} {\rm Bi} \left (l^{2/3} 2^{1/3} 
[1-\xi] \right )    
\end{equation}

\end{document}